\documentclass[a4paper,10pt]{amsart}

\usepackage{amsmath,amssymb,amsfonts,latexsym}

\newtheorem{theorem}{Theorem}[section]
\newtheorem{lemma}[theorem]{Lemma}
\newtheorem{corollary}[theorem]{Corollary}

\begin{document}

\title[On a Generalization for Tribonacci Quaternions]{On a Generalization for Tribonacci Quaternions}

\author[G. Cerda-Morales]{Gamaliel Cerda-Morales}
\address{Instituto de Matem\'maticas, Pontificia Universidad Cat\'olica de Valpara\'iso, Blanco Viel 596, Valpara\'iso, Chile.}
\email{gamaliel.cerda.m@mail.pucv.cl}


\begin{abstract}
Let $V_{n}$ denote the third order linear recursive sequence defined by the initial values $V_{0}$, $V_{1}$ and $V_{2}$ and the recursion $V_{n}=rV_{n-1}+sV_{n-2}+tV_{n-3}$ if $n\geq 3$, where $r$, $s$, and $t$ are real constants. The $\{V_{n}\}_{n\geq0}$ are generalized Tribonacci numbers and reduce to the usual Tribonacci numbers when $r=s=t=1$ and to the $3$-bonacci numbers when $r=s=1$ and $t=0$. In this study, we introduced a quaternion sequence which has not been introduced before. We show that the new quaternion sequence that we introduced includes the previously introduced Tribonacci, Padovan, Narayana and Third order Jacobsthal quaternion sequences. We obtained the Binet formula, summation formula and the norm value for this new quaternion sequence.

\vspace{2mm}

\noindent\textsc{2010 Mathematics Subject Classification.} 11R52, 11B37, 11B39, 11B83.

\vspace{2mm}

\noindent\textsc{Keywords and phrases.} Quaternion, Generalized Tribonacci sequence, Narayana sequence, Third order Jacobsthal sequence.

\end{abstract}



\maketitle


\section{Generalized Tribonacci sequence}
We consider the generalized Tribonacci Sequence, $\{V_{n}(V_{0},V_{1},V_{2};r,s,t)\}_{n\geq0}$, or briefly $\{V_{n}\}_{n\geq0}$, defined as follows:
\begin{equation}\label{eq:1}
V_{n}=rV_{n-1}+sV_{n-2}+tV_{n-3},\ n\geq3,
\end{equation}
where $V_{0}=a$, $V_{1}=b$, $V_{2}=c$ are arbitrary integers and $r$, $s$, $t$, are real numbers.

This sequence has been studied by Shannon and Horadam \cite{Sha}, Yalavigi \cite{Ya} and Pethe \cite{Pe}. If we set $r=s=t=1$ and $V_{0}=0=V_{1}$, $V_{2}=1$, then $\{V_{n}\}$ is the well-known Tribonacci sequence which has been considered extensively (see, for example, \cite{Fe}). 

As the elements of this Tribonacci-type number sequence provide third order iterative relation, its characteristic equation is $x^{3}-rx^{2}-sx-t=0$, whose roots are $\alpha=\alpha(r,s,t)=\frac{r}{3}+A+B$, $\omega_{1}=\frac{r}{3}+\epsilon A+\epsilon^{2} B$ and $\omega_{2}=\frac{r}{3}+\epsilon^{2}A+\epsilon B$, where $$A=\left(\frac{r^{3}}{27}+\frac{rs}{6}+\frac{t}{2}+\sqrt{\Delta}\right)^{\frac{1}{3}},\ B=\left(\frac{r^{3}}{27}+\frac{rs}{6}+\frac{t}{2}-\sqrt{\Delta}\right)^{\frac{1}{3}},$$ with $\Delta=\Delta(r,s,t)=\frac{r^{3}t}{27}-\frac{r^{2}s^{2}}{108}+\frac{rst}{6}-\frac{s^{3}}{27}+\frac{t^{2}}{4}$ and $\epsilon=-\frac{1}{2}+\frac{i\sqrt{3}}{2}$. 

In this paper, $\Delta(r,s,t)>0$, then the equation (\ref{eq:1}) has one real and two nonreal solutions, the latter being conjugate complex. Thus, the Binet formula for the generalized Tribonacci numbers can be expressed as:
\begin{equation}\label{eq:2}
V_{n}=\frac{P\alpha^{n}}{(\alpha-\omega_{1})(\alpha-\omega_{2})}-\frac{Q\omega_{1}^{n}}{(\alpha-\omega_{1})(\omega_{1}-\omega_{2})}+\frac{R\omega_{2}^{n}}{(\alpha-\omega_{2})(\omega_{1}-\omega_{2})},
\end{equation}
where $P=c-(\omega_{1}+\omega_{2})b+\omega_{1}\omega_{2}a$, $Q=c-(\alpha+\omega_{2})b+\alpha\omega_{2}a$ and $R=c-(\alpha+\omega_{1})b+\alpha\omega_{1}a$.

In fact, the generalized Tribonacci sequence is the generalization of the well-known sequences like Tribonacci, Padovan, Narayana and third order Jacobsthal. For example, $\{V_{n}(0,0,1;1,1,1)\}_{n\geq0}$, $\{V_{n}(0,1,0;0,1,1)\}_{n\geq0}$, are Tribonacci and Padovan sequences, respectively. The Binet formula for the generalized Tribonacci sequence is expressed as follows:
\begin{lemma}
The Binet formula for the generalized Tribonacci sequence is:
\begin{equation}\label{eq:3}
V_{n}=cU_{n}+(bs+at)U_{n-1}+btU_{n-2},
\end{equation}
where $a$, $b$ and $c$ are initial values and
\begin{equation}\label{eq:4}
U_{n}=\frac{\alpha^{n}}{(\alpha-\omega_{1})(\alpha-\omega_{2})}-\frac{\omega_{1}^{n}}{(\alpha-\omega_{1})(\omega_{1}-\omega_{2})}+\frac{\omega_{2}^{n}}{(\alpha-\omega_{2})(\omega_{1}-\omega_{2})}.
\end{equation}
\end{lemma}
\begin{proof}
The validity of this formula can be confirmed using the recurrence relation. Furthermore, $\{U_{n}\}_{n\geq0}=\{V_{n}(0,0,1;r,s,t)\}_{n\geq0}$.
\end{proof}

On the other hand, Horadam \cite{Ho} introduced the $n$-th Fibonacci and the $n$-th Lucas quaternion as follow:
\begin{equation}\label{equa:1}
Q_{n}=F_{n}+iF_{n+1}+jF_{n+2}+kF_{n+3},
\end{equation}
\begin{equation}\label{equa:2}
K_{n}=L_{n}+iL_{n+1}+jL_{n+2}+kL_{n+3},
\end{equation}
respectively. Here $F_{n}$ and $L_{n}$ are the $n$-th Fibonacci and Lucas numbers, respectively. Furthermore, the basis $i,j,k$ satisfies the following rules:
\begin{equation}\label{equa:3}
i^{2}=j^{2}=k^{2}=-1,\ ijk=-1.
\end{equation}

Note that the rules (\ref{equa:3}) imply $ij=-ji=k$, $jk=-kj=i$ and $ki=-ik=j$. In general, a quaternion is a hyper-complex number and is defined by $q=q_{0}+iq_{1}+jq_{2}+kq_{3}$, where $i,j,k$ are as in (\ref{equa:3}) and $q_{0},q_{1},q_{2}$ and $q_{3}$ are real numbers. Note that we can write $q=q_{0}+u$ where $u= iq_{1}+jq_{2}+kq_{3}$. The conjugate of the quaternion $q$ is denoted by $q^{*}=q_{0}-u$. The norm of a quaternion $q$ is defined by $Nr^{2}(q)=q_{0}^{2}+q_{1}^{2}+q_{2}^{2}+q_{3}^{2}$. 

In this paper we introduce and study the generalized Tribonacci quaternions. We show that the new quaternion sequence that we introduced includes the previously introduced Tribonacci, Padovan, Narayana and Third order Jacobsthal quaternion sequences. We obtained the Binet formula and calculated a quadratic identity, summation formula and the norm value for this new quaternion sequence. Furthermore, we describe their properties also using a matrix representation.

\section{Generalized Tribonacci quaternions}
In this section, we define a new relation which generalizes all the third order recurrence relations which have been studied so far. We present some quadratic identities and some results corresponding to the results of previously studied sequences. 

First, for $n\geq 2$ and $p$, $q$ integers, using $W_{n}=pW_{n-1}+qW_{n-2}$, $W_{0}=a$ and $W_{1}=b$, Halici and Karata\c{s} in \cite{Ha2} showed the presence of the following iterative relation $W_{n}=bT_{n}+aqT_{n-1}$, where $T_{n}=pT_{n-1}+qT_{n-2}$, $T_{0}=0$ and $T_{1}=1$. The above formula is a type of generalization for Fibonacci numbers. Furthermore, the authors made a new generalization for the Fibonacci quaternions with the help of this relation, which is expressed below: 
\begin{equation}\label{eq:5}
Q_{w,n}=W_{n}+W_{n+1}i+W_{n+2}j+W_{n+3}k,\ n\geq0,
\end{equation}
where $W_{n}$ is the $n$-th Horadam number.

Now, let us propose a new type of quaternion similar to generalized Tribonacci:
\begin{equation}\label{eq:6}
Q_{v,n}=V_{n}+V_{n+1}i+V_{n+2}j+V_{n+3}k,\ n\geq0.
\end{equation}
Here, $V_{n}$ is the $n$-th generalized Tribonacci number defined by Eq. (\ref{eq:1}). This new quaternion can be called generalized Tribonacci quaternion due to its coefficients. The following iterative relation is obtained after some basic calculations:
\begin{equation}\label{eq:7}
Q_{v,n}=rQ_{v,n-1}+sQ_{v,n-2}+tQ_{v,n-3},\ n\geq3.
\end{equation}

The Eq. (\ref{eq:6}) generalizes different quaternionic sequences, which are known as quaternions similar to Tribonacci in the literature. Two of the most well-known of these are Narayana and Third order Jacobsthal quaternions. When we calculate the initial values of the sequence in Eq. (\ref{eq:7}), we find:
\begin{equation}\label{eq:8}
\left\{ 
\begin{array}{c}
Q_{v,0}=a+bi+cj+(rc+sb+ta)k, \\ 
Q_{v,1}=b+ci+(rc+sb+ta)j+((r^{2}+s)c+(t+rs)b+rta)k, \\
Q_{v,2}=c+(rc+sb+ta)i+((r^{2}+s)c+(t+rs)b+rta)j+V_{5}k.
\end{array}
\right.
\end{equation}
where $V_{5}=(r^{3}+2rs+t)c+(r^{2}s+s^{2}+rt)b+(r^{2}t+st)a$.

It must be noted that when $V_{n}(0,1,1;1,1,2)$ are taken instead of $V_{n}(r,s,t)$ in Eqs. (\ref{eq:8}), $Q_{v,0}=i+j+2k$, $Q_{v,0}=1+i+2j+5k$ and $Q_{v,0}=1+2i+5j+9k$ is obtained. Here $Q_{v,n}=QJ_{n}^{(3)}$ are Third order Jacobsthal quaternions found by Cerda-Morales in \cite{Ce}. For this case, the roots of the characteristic equation of the sequence are $x=2$ and $x_{1,2}=\frac{-1\pm i\sqrt{3}}{2}$. Note that the latter two are the complex conjugate cube roots of unity. In \cite{Ce}, with the help of the roots above, Binet formula for Third order Jacobsthal quaternions was given as follows:
\begin{equation}\label{eq:9}
JQ_{n}^{(3)}=\frac{1}{7}\left[2^{n+1}\underline{2}-\left(1+\frac{2i\sqrt{3}}{3}\right)x_{1}^{n}\underline{x_{1}}-\left(1-\frac{2i\sqrt{3}}{3}\right)x_{2}^{n}\underline{x_{2}}\right],
\end{equation}
where $x_{1,2}$ are the solutions of the equation $t^{2}+t+1=0$, and $\underline{2}=1+2i+4j+8k$, $\underline{x_{1}}=1+x_{1}i+x_{1}^{2}j+k$ and $\underline{x_{2}}=1+x_{2}i+x_{2}^{2}j+k$.

In the following theorem, we derive the Binet formula for generalized Tribonacci quaternions.
\begin{theorem}
Binet formula for generalized Tribonacci quaternions is
\begin{equation}\label{eq:10}
Q_{v,n}=\frac{P\underline{\alpha}\alpha^{n}}{(\alpha-\omega_{1})(\alpha-\omega_{2})}-\frac{Q\underline{\omega_{1}}\omega_{1}^{n}}{(\alpha-\omega_{1})(\omega_{1}-\omega_{2})}+\frac{R\underline{\omega_{2}}\omega_{2}^{n}}{(\alpha-\omega_{2})(\omega_{1}-\omega_{2})},
\end{equation}
where $P$, $Q$ and $R$ as in Eq. (\ref{eq:2}), and $\underline{\alpha}=1+\alpha i+\alpha^{2} i+\alpha^{3} k$, $\underline{\omega_{1}}=1+\omega_{1} i+\omega_{1}^{2} i+\omega_{1}^{3} k$ and $\underline{\omega_{2}}=1+\omega_{2} i+\omega_{2}^{2} i+\omega_{2}^{3} k$.
\end{theorem}
\begin{proof}
Using the definition of generalized Tribonacci quaternions $\{Q_{v,n}\}_{n\geq0}$, the following equation can be written:
\begin{align*}
Q_{v,n}&=\frac{P\alpha^{n}}{(\alpha-\omega_{1})(\alpha-\omega_{2})}-\frac{Q\omega_{1}^{n}}{(\alpha-\omega_{1})(\omega_{1}-\omega_{2})}+\frac{R\omega_{2}^{n}}{(\alpha-\omega_{2})(\omega_{1}-\omega_{2})}\\
&\ \ + \left(\frac{P\alpha^{n+1}}{(\alpha-\omega_{1})(\alpha-\omega_{2})}-\frac{Q\omega_{1}^{n+1}}{(\alpha-\omega_{1})(\omega_{1}-\omega_{2})}+\frac{R\omega_{2}^{n+1}}{(\alpha-\omega_{2})(\omega_{1}-\omega_{2})}\right)i\\
&\ \ + \left(\frac{P\alpha^{n+2}}{(\alpha-\omega_{1})(\alpha-\omega_{2})}-\frac{Q\omega_{1}^{n+2}}{(\alpha-\omega_{1})(\omega_{1}-\omega_{2})}+\frac{R\omega_{2}^{n+2}}{(\alpha-\omega_{2})(\omega_{1}-\omega_{2})}\right)j\\
&\ \ + \left(\frac{P\alpha^{n+3}}{(\alpha-\omega_{1})(\alpha-\omega_{2})}-\frac{Q\omega_{1}^{n+3}}{(\alpha-\omega_{1})(\omega_{1}-\omega_{2})}+\frac{R\omega_{2}^{n+3}}{(\alpha-\omega_{2})(\omega_{1}-\omega_{2})}\right)k.
\end{align*}
Then, some basic calculations are made and Eq. (\ref{eq:10}) is found.
\end{proof}

In the following table, we give the formula $\alpha(r,s,t)$ for some special third order sequences.
\begin{table}[ht] 
\caption{Formulas for $\alpha(r,s,t)$ according to different values of $r,s,t$.} 
\centering      
\begin{tabular}{c c c}  
\hline                      
Sequences & $(V_{0},V_{1},V_{2};r,s,t)$ & $\alpha(r,s,t)$\\ [0.5ex] 
\hline    
Narayana & $\left(0,1,1;1,0,1\right)$  & $\frac{1}{3}+\sqrt[3]{\frac{29}{54}+\frac{\sqrt{93}}{18}}+\sqrt[3]{\frac{29}{54}-\frac{\sqrt{93}}{18}}$\\                
Tribonacci-Lucas & $\left(0,0,1;1,1,1\right)$  & $\frac{1}{3}+\sqrt[3]{\frac{19}{27}+\frac{\sqrt{33}}{9}}+\sqrt[3]{\frac{19}{27}-\frac{\sqrt{33}}{9}}$\\  
Padovan-Perrin & $\left(0,1,0;0,1,1\right)$  & $\sqrt[3]{\frac{1}{2}+\frac{\sqrt{69}}{18}}+\sqrt[3]{\frac{1}{2}-\frac{\sqrt{69}}{18}}$\\  
Third order Jacobsthal  &  $\left(0,1,1;1,1,2\right)$ & $\frac{4}{3}$ \\ [1ex]       
\hline     
\end{tabular} 
\label{table:1}  
\end{table}

The values given in the table are the same with the values given in some previous studies (see, for example, \cite{Di,Fe,Pe,Sha}). In other words, the formula we give is a generalization. For example, the Fibonacci-Narayana quaternions given in the first row of the table were also defined by Flaut and Shpakivskyi \cite{Fla1} by giving suitable initial values and by listing its properties. Also, Third order Jacobsthal quaternions in the fourth row of the table were defined by Cerda-Morales \cite{Ce}, who examined its properties. As a result, here we generalized different Tribonacci-like quaternion sequences in one formula.

In the following theorem we present the generating function for generalized Tribonacci quaternions.
\begin{theorem}
The generating function for the generalized Tribonacci quaternion $Q_{v,n}$ is
\begin{equation}\label{eq:11}
g(x)=\frac{Q_{v,0}+(Q_{v,1}-rQ_{v,0})x+(Q_{v,2}-rQ_{v,1}-sQ_{v,0})x^{2}}{1-rx-sx^{2}-tx^{3}}
\end{equation}
\end{theorem}
\begin{proof}
Assuming that the generating function of the generalized Tribonacci quaternion sequence $\{Q_{v,n}\}_{n\geq 0}$ has the form $g(x)=\sum_{n\geq0}Q_{v,n}x^{n}$, we obtain that
\begin{align*}
(1-rx-&sx^{2}-tx^{3})\sum_{n\geq0}Q_{v,n}x^{n}\\
&=Q_{v,0}+Q_{v,1}x+Q_{v,2}x^{2}+Q_{v,3}x^{3}+\cdots\\
&\ \ -rQ_{v,0}x-rQ_{v,1}x^{2}-rQ_{v,2}x^{3}-rQ_{v,3}x^{4}-\cdots\\
&\ \ -sQ_{v,0}x^{2}-sQ_{v,1}x^{3}-sQ_{v,2}x^{4}-sQ_{v,3}x^{5}-\cdots\\
&\ \ -tQ_{v,0}x^{3}-tQ_{v,1}x^{4}-tQ_{v,2}x^{5}-tQ_{v,3}x^{6}-\cdots\\
&=Q_{v,0}+(Q_{v,1}-rQ_{v,0})x+(Q_{v,2}-rQ_{v,1}-sQ_{v,0})x^{2},
\end{align*}
since $Q_{v,n}=rQ_{v,n-1}+sQ_{v,n-2}+tQ_{v,n-3}$, $n\geq3$ and the coefficients of $x^{n}$ for $n\geq 3$ are equal with zero. Then, we get $$\sum_{n\geq0}Q_{v,n}x^{n}=\frac{Q_{v,0}+(Q_{v,1}-rQ_{v,0})x+(Q_{v,2}-rQ_{v,1}-sQ_{v,0})x^{2}}{1-rx-sx^{2}-tx^{3}}.$$ The theorem is proved.
\end{proof}

In the following table, we examine some special cases of generating functions given in Eq. (\ref{eq:11}). 
\begin{table}[ht] 
\caption{Generating functions according to initial values.} 
\centering      
\begin{tabular}{ c c}  
\hline                        
Narayana quaternions & $\frac{x+i+(1+x^{2})j+(1+x+x^{2})k}{1-x-x^{3}}$\\                
Tribonacci quaternions & $\frac{x^{2}+xi+j+(1+x+x^{2})k}{1-x-x^{2}-x^{3}}$  \\  
Padovan-Perrin quaternions & $\frac{x+i+(x+x^{2})j+(1+x)k}{1-x^{2}-x^{3}}$  \\  
Third order Jacobsthal quaternions  &  $\frac{x+i+(1+x+2x^{2})j+(2+3x+2x^{2})k}{1-x-x^{2}-2x^{3}}$ \\ [1ex]       
\hline     
\end{tabular} 
\label{table:2}  
\end{table}

It must be restated that the calculations in the table above were made according to the initial condition of each sequence. The calculation in the first row was made by Flaut and Shpakivskyi \cite{Fla1}. The sequence in the fourth row of the table was studied by Cerda-Morales in \cite{Ce}.

Now, let us write the formula which gives the summation of the first $n$ generalized Tribonacci numbers and quaternions.
\begin{lemma}
For every integer $n\geq 0$, we have:
\begin{equation}\label{eq:12}
\sum_{l=0}^{n}V_{l}=\frac{1}{\delta(r,s,t)}\left(V_{n+2}+(1-r)V_{n+1}+tV_{n}+(r+s-1)a+(r-1)b-c\right),
\end{equation}
where $\delta=\delta(r,s,t)=r+s+t-1$ and $V_{n}$ denote the $n$-th term of the generalized Tribonacci numbers.
\end{lemma}
\begin{proof}
By the principal of mathematics induction since $\delta(r,s,t)V_{0}$ is equal to $V_{2}+(1-r)V_{1}+tV_{0}+(r+s-1)a+(r-1)b-c=\delta(r,s,t)a$, then the statement is true when $n=0$. Assume the given statement is true for all integer $n$. Then,
\begin{align*}
\delta(r,s,t)&\sum_{l=0}^{n+1}V_{l}=\delta(r,s,t)\sum_{l=0}^{n}V_{l}+\delta(r,s,t)V_{n+1}\\
&=V_{n+2}+(1-r)V_{n+1}+tV_{n}+(r+s-1)a+(r-1)b-c\\
&\ \ +(r+s+t-1)V_{n+1}\\
&=V_{n+3}+(1-r)V_{n+2}+tV_{n+1}+(r+s-1)a+(r-1)b-c.
\end{align*}
Thus, by the principal of mathematics induction the formula holds for every integer $n\geq0$.
\end{proof}
\begin{theorem}
The summation formula for generalized Tribonacci quaternions is as follows:
\begin{equation}\label{eq:13}
\sum_{l=0}^{n}Q_{v,l}=\frac{1}{\delta(r,s,t)}\left(Q_{v,n+2}+(1-r)Q_{v,n+1}+tQ_{v,n}+\omega(r,s,t)\right),
\end{equation}
where $\omega(r,s,t)=\lambda+i(\lambda-\delta a)+j(\lambda-\delta(a+b)+k(\lambda-\delta(a+b+c)$ and $\lambda=\lambda(r,s,t)$ is equal to $(r+s-1)a+(r-1)b-c$.
\end{theorem}
\begin{proof}
Using Eq. (\ref{eq:6}), we have
\begin{align*}
\sum_{l=0}^{n}Q_{v,l}&=\sum_{l=0}^{n}V_{l}+i\sum_{l=0}^{n}V_{l+1}+j\sum_{l=0}^{n}V_{l+2}+k\sum_{l=0}^{n}V_{l+3}\\
&=(V_{0}+V_{1}+V_{2}+\cdots+V_{n})+i(V_{1}+V_{2}+V_{3}+\cdots+V_{n+1})\\
&\ \ +j(V_{2}+V_{3}+V_{4}+\cdots+V_{n+2})+k(V_{3}+V_{4}+V_{5}+\cdots+V_{n+3}).
\end{align*}
Since from Eq. (\ref{eq:12}) and using the notation $\lambda(r,s,t)=(r+s-1)a+(r-1)b-c$, we can write 
\begin{align*}
\delta(r,s,t)\sum_{l=0}^{n}Q_{v,l}&=V_{n+2}+(1-r)V_{n+1}+tV_{n}+\lambda(r,s,t)\\
&\ \ +i\left(V_{n+3}+(1-r)V_{n+2}+tV_{n+1}+\lambda(r,s,t)-\delta a\right)\\
&\ \ +j\left(V_{n+4}+(1-r)V_{n+3}+tV_{n+2}+\lambda(r,s,t)-\delta(a+b)\right)\\
&\ \ +k\left(V_{n+5}+(1-r)V_{n+4}+tV_{n+3}+\lambda(r,s,t)-\delta(a+b+c)\right)\\
&=Q_{v,n+2}+(1-r)Q_{v,n+1}+tQ_{v,n}+\omega(r,s,t),
\end{align*}
where $\omega(r,s,t)=\lambda+i(\lambda-\delta a)+j(\lambda-\delta(a+b)+k(\lambda-\delta(a+b+c)$. Finally, $$\sum_{l=0}^{n}Q_{v,l}=\frac{1}{\delta(r,s,t)}\left(Q_{v,n+2}+(1-r)Q_{v,n+1}+tQ_{v,n}+\omega(r,s,t)\right).$$
The theorem is proved.
\end{proof}

The summation formula in Eq. (\ref{eq:13}) gives the sum of the elements in the quaternion sequences which have been found in the studies conducted so far. This can be seen in the following table.
\begin{table}[ht] 
\caption{Summation formulas according to initial values.} 
\centering      
\begin{tabular}{ c c  }  
\hline                        
Narayana quaternions & $Q_{v,n+3}-(1+i+2j+3k)$\\                
Tribonacci quaternions & $\frac{1}{2}\left(Q_{v,n+2}+Q_{v,n}-(1+i+j+3k)\right)$  \\  
Padovan-Perrin quaternions & $Q_{v,n+5}-(1+i+2j+2k)$  \\  
Third order Jacobsthal quaternions  &  $\frac{1}{3}\left(Q_{v,n+2}+2Q_{v,n}-(1+i+4j+7k)\right)$ \\ [1ex]       
\hline     
\end{tabular} 
\label{table:3}  
\end{table}

Now, we present the formula which gives the norms for generalized Tribonacci quaternions. If we use the definition  norm, then we obtain $Nr^{2}(Q_{v,n})=\sum_{l=0}^{3}V_{n+l}^{2}$. Moreover, by the Binet formula (\ref{eq:2}) we have $$\phi V_{n}=(\omega_{1}-\omega_{2})P\alpha^{n}-(\alpha-\omega_{2})Q\omega_{1}^{n}+(\alpha-\omega_{1})R\omega_{2}^{n},$$ where $\phi=\phi(\alpha,\omega_{1},\omega_{2})=(\alpha-\omega_{1})(\alpha-\omega_{2})(\omega_{1}-\omega_{2})$. Then, 
\begin{align*}
\phi^{2} V_{n}^{2}&=(\omega_{1}-\omega_{2})^{2}P^{2}\alpha^{2n}+(\alpha-\omega_{2})^{2}Q^{2}\omega_{1}^{2n}+(\alpha-\omega_{1})^{2}R^{2}\omega_{2}^{2n}\\
&\ \ -2(\omega_{1}-\omega_{2})(\alpha-\omega_{2})PQ(\alpha\omega_{1})^{n}+2(\omega_{1}-\omega_{2})(\alpha-\omega_{1})PR(\alpha\omega_{2})^{n}\\
&\ \ -2(\alpha-\omega_{1})(\alpha-\omega_{2})QR(\omega_{1}\omega_{2})^{n}
\end{align*}
and 
\begin{align*}
\phi^{2} Nr^{2}&(Q_{v,n})=\phi^{2} (V_{n}^{2}+V_{n+1}^{2}+V_{n+2}^{2}+V_{n+3}^{2})\\
&=(\omega_{1}-\omega_{2})^{2}P^{2}\overline{\alpha}\alpha^{2n}+(\alpha-\omega_{2})^{2}Q^{2}\overline{\omega_{1}}\omega_{1}^{2n}+(\alpha-\omega_{1})^{2}R^{2}\overline{\omega_{2}}\omega_{2}^{2n}\\
&\ \ -2(\omega_{1}-\omega_{2})(\alpha-\omega_{2})PQ \underline{\alpha\omega_{1}}(\alpha\omega_{1})^{n}-2(\omega_{1}-\omega_{2})(\omega_{1}-\alpha)PR\underline{\alpha\omega_{2}}(\alpha\omega_{2})^{n}\\
&\ \ -2(\alpha-\omega_{1})(\alpha-\omega_{2})QR\underline{\omega_{1}\omega_{2}}(\omega_{1}\omega_{2})^{n},
\end{align*}
where $\overline{\alpha}=1+\alpha^{2}+\alpha^{4}+\alpha^{6}$, $\overline{\omega_{1,2}}=1+\omega_{1,2}^{2}+\omega_{1,2}^{4}+\omega_{1,2}^{6}$, $\underline{\alpha \omega_{1,2}}=1+ \alpha \omega_{1,2}+ (\alpha \omega_{1,2})^{2}+ (\alpha \omega_{1,2})^{3}$ and $\underline{\omega_{1}\omega_{2}}=1+\omega_{1}\omega_{2}+(\omega_{1}\omega_{2})^{2}+(\omega_{1}\omega_{2})^{3}$. 

Then, we obtain
\begin{theorem}
The norm value for generalized Tribonacci quaternions is given with the following formula:
\begin{equation}\label{eq:18}
Nr^{2}(Q_{v,n})=\frac{1}{\phi^{2}}\left( 
\begin{array}{c}
(\omega_{1}-\omega_{2})^{2}P^{2}\overline{\alpha}\alpha^{2n}+(\alpha-\omega_{2})^{2}Q^{2}\overline{\omega_{1}}\omega_{1}^{2n} \\ 
+(\alpha-\omega_{1})^{2}R^{2}\overline{\omega_{2}}\omega_{2}^{2n}-2K%
\end{array}%
\right)
\end{equation}
where $K=(\omega_{1}-\omega_{2})(\alpha-\omega_{2})PQ \underline{\alpha\omega_{1}}(\alpha\omega_{1})^{n}+(\omega_{1}-\omega_{2})(\omega_{1}-\alpha)PR\underline{\alpha\omega_{2}}(\alpha\omega_{2})^{n}+(\alpha-\omega_{1})(\alpha-\omega_{2})QR\underline{\omega_{1}\omega_{2}}(\omega_{1}\omega_{2})^{n}$.
\end{theorem}

Now, we will give binomial summation of generalized Tribonacci quaternions as follows:
\begin{theorem}
For $n\geq 0$, we have the equality
\begin{equation}\label{eq:19}
Q_{v,3n}=\sum_{l=0}^{n}\sum_{m=0}^{l}\binom{n}{l}\binom{l}{m}r^{m}s^{l-m}t^{n-l}Q_{v,l+m}.
\end{equation}
\end{theorem}
\begin{proof}
Let $\alpha$ stand for a root of the characteristic equation of Eq. (\ref{eq:1}). Then, we have $\alpha^{3}=r\alpha^{2}+s\alpha+t$ and we can write by considering binomial expansion with $t\neq0$:
\begin{align*}
\left(\frac{\alpha^{3}}{t}\right)^{n}&=\sum_{l=0}^{n}\binom{n}{l}\left(\frac{\alpha^{3}}{t}-1\right)^{l}=\sum_{l=0}^{n}\binom{n}{l}\left(\frac{r}{t}\alpha^{2}+\frac{s}{t}\alpha\right)^{l}\\
&=\sum_{l=0}^{n}\binom{n}{l}\sum_{m=0}{l}\binom{l}{m}\left(\frac{r}{t}\alpha^{2}\right)^{m}\left(\frac{s}{t}\alpha\right)^{l-m}\\
&=\sum_{l=0}^{n}\sum_{m=0}^{l}\binom{n}{l}\binom{l}{m}\left(\frac{r}{s}\right)^{m}\left(\frac{s}{t}\right)^{l}\alpha^{l+m}.
\end{align*}
If we replace to $\omega_{1}$ and $\omega_{2}$ by $\alpha$ and rearrange, then we obtain 
\begin{align*}
\frac{Q_{v,3n}}{t^{n}}&=\frac{P\underline{\alpha}}{(\alpha-\omega_{1})(\alpha-\omega_{2})}\frac{\alpha^{3n}}{t^{n}}-\frac{Q\underline{\omega_{1}}}{(\alpha-\omega_{1})(\omega_{1}-\omega_{2})}\frac{\omega_{1}^{3n}}{t^{n}}+\frac{R\underline{\omega_{2}}}{(\alpha-\omega_{2})(\omega_{1}-\omega_{2})}\frac{\omega_{2}^{3n}}{t^{n}}\\
&=\sum_{l=0}^{n}\sum_{m=0}^{l}\binom{n}{l}\binom{l}{m}\left(\frac{r}{s}\right)^{m}\left(\frac{s}{t}\right)^{l}Q_{v,l+m}.
\end{align*}
where $P$, $Q$ and $R$ as in Eq. (\ref{eq:2}), and $\underline{\alpha}$, $\underline{\omega_{1}}$ and $\underline{\omega_{2}}$ as in Eq. (\ref{eq:10}).
\end{proof}

\section{Matrix Representation of Generalized Tribonacci Quaternions}
For our purposes, the most useful technique for generating $\{V_{n}\}$ is by means of what we call the $S$-matrix which has been defined and used in \cite{Sha} and is a generalization of the $R$-matrix defined in \cite{Wa}. The $S$-matrix is defined as
\begin{equation}\label{eq:14}
\left[ 
\begin{array}{c}
V_{n+2} \\ 
V_{n+1} \\ 
V_{n}%
\end{array}%
\right]=\left[ 
\begin{array}{ccc}
r & s & t \\ 
1 & 0 & 0 \\ 
0 & 1 & 0
\end{array}%
\right]^{n}\left[ 
\begin{array}{c}
V_{2} \\ 
V_{1} \\ 
V_{0}%
\end{array}%
\right]
\end{equation}
and
\begin{equation}\label{eq:15}
S^{n}=\left[ 
\begin{array}{ccc}
r & s & t \\ 
1 & 0 & 0 \\ 
0 & 1 & 0
\end{array}%
\right]^{n}=\left[ 
\begin{array}{ccc}
U_{n+2} & sU_{n+1}+tU_{n} & tU_{n+1} \\ 
U_{n+1} & sU_{n}+tU_{n-1} & tU_{n} \\ 
U_{n} & sU_{n-1}+tU_{n-2} & tU_{n-1}
\end{array}%
\right],
\end{equation}
where $U_{-1}=\frac{1}{t}$ and $U_{-2}=-\frac{s}{t^{2}}$.

Now, let us define the following matrix as
\begin{equation}\label{eq:16}
Q_{S}=\left[ 
\begin{array}{ccc}
Q_{v,4} & sQ_{v,3}+tQ_{v,2} & tQ_{v,3} \\ 
Q_{v,3} & sQ_{v,2}+tQ_{v,1} & tQ_{v,2} \\ 
Q_{v,2} & sQ_{v,1}+tQ_{v,0} & tQ_{v,1}
\end{array}%
\right].
\end{equation}
This matrix can be called as the generalized Tribonacci quaternion matrix. Then, we can give the next theorem to the $Q_{S}$-matrix.

\begin{theorem}
If $Q_{v,n}$ be the $n$-th generalized Tribonacci quaternion. Then, for $n\geq0$:
\begin{equation}\label{eq:17}
Q_{S}\cdot\left[ 
\begin{array}{ccc}
r & s & t \\ 
1 & 0 & 0 \\ 
0 & 1 & 0
\end{array}%
\right]^{n}=\left[ 
\begin{array}{ccc}
Q_{v,n+4} & sQ_{v,n+3}+tQ_{v,n+2} & tQ_{v,n+3} \\ 
Q_{v,n+3} & sQ_{v,n+2}+tQ_{v,n+1} & tQ_{v,n+2} \\ 
Q_{v,n+2} & sQ_{v,n+1}+tQ_{v,n} & tQ_{v,n+1}%
\end{array}%
\right].
\end{equation}
\end{theorem}
\begin{proof}
(By induction on $n$) If $n=0$, then the result is obvious. Now, we suppose it is true for $n=m$, that is
$$Q_{S}\cdot S^{m}=\left[ 
\begin{array}{ccc}
Q_{v,m+4} & sQ_{v,m+3}+tQ_{v,m+2} & tQ_{v,m+3} \\ 
Q_{v,m+3} & sQ_{v,m+2}+tQ_{v,m+1} & tQ_{v,m+2} \\ 
Q_{v,m+2} & sQ_{v,m+1}+tQ_{v,m} & tQ_{v,m+1}%
\end{array}%
\right].$$
Using the Eq. (\ref{eq:7}), for $m\geq 0$, $Q_{v,m+3}=rQ_{v,m+2}+sQ_{v,m+1}+tQ_{v,m}$. Then, by induction hypothesis
\begin{align*}
Q_{S}\cdot S^{m+1}&=\left(Q_{S}\cdot S^{m}\right)\cdot S\\
&=\left[ 
\begin{array}{ccc}
Q_{v,m+4} & sQ_{v,m+3}+tQ_{v,m+2} & tQ_{v,m+3} \\ 
Q_{v,m+3} & sQ_{v,m+2}+tQ_{v,m+1} & tQ_{v,m+2} \\ 
Q_{v,m+2} & sQ_{v,m+1}+tQ_{v,m} & tQ_{v,m+1}%
\end{array}%
\right]\left[ 
\begin{array}{ccc}
r & s & t \\ 
1 & 0 & 0 \\ 
0 & 1 & 0%
\end{array}%
\right]\\
&=\left[ 
\begin{array}{ccc}
Q_{v,m+5} & sQ_{v,m+4}+tQ_{v,m+3} & tQ_{v,m+4} \\ 
Q_{v,m+4} & sQ_{v,m+3}+tQ_{v,m+2} & tQ_{v,m+3} \\ 
Q_{v,m+3} & sQ_{v,m+2}+tQ_{v,m+1} & tQ_{v,m+2}%
\end{array}%
\right].
\end{align*}
Hence, the Eq. (\ref{eq:17}) holds for all $n\geq0$.
\end{proof}

\begin{corollary}
For $n\geq0$, 
\begin{equation}
Q_{w,n+2}=Q_{w,2}U_{n+2}+(sQ_{v,1}+tQ_{v,0})U_{n}+tQ_{v,1}U_{n}.
\end{equation}
\end{corollary}
\begin{proof}
The proof can be easily seen by the coefficient (3,1) of the matrix $Q_{S}\cdot S^{n}$ and the Eq. (\ref{eq:15}).
\end{proof}

\section{Conclusion}
This study examines and studied Tribonacci-type quaternion sequences with the help of a simple and general formula. For this purpose, generalized Tribonacci sequence $\{V_{n}\}_{n\geq0}$ was used. Generalized Tribonacci sequence was examined in detail particularly in section 1, and it was shown that this sequence is used to generalize all the third order linear recurrence relations on quaternions. In this study, Binet formulas, generating functions, some summation formulas and norm values of all quaternion sequences were obtained. As a result, all the formulas in the literature were given with the help of only one formula. Quaternions have great importance as they are used in quantum physics, applied mathematics, graph theory and differential equations. Thus, in our future studies we plan to examine generalized Tribonacci octonions and their key features.


\end{document}